\documentclass{amsart}
\usepackage{amssymb}
\usepackage{color}
\usepackage{hyperref}
\usepackage{tikz-cd}
\usepackage{rotating}

\newtheorem{theorem}{Theorem}[section]
\newtheorem{lemma}[theorem]{Lemma}
\newtheorem{proposition}[theorem]{Proposition}
\newtheorem{corollary}[theorem]{Corollary}

\theoremstyle{definition}
\newtheorem{definition}[theorem]{Definition}

\newtheorem{remark}[theorem]{Remark}
\numberwithin{equation}{section}
\usepackage[all]{xy}
\DeclareMathOperator{\Po}{Po}
\DeclareMathOperator{\Ost}{Ost}
\DeclareMathOperator{\Cl}{Cl}

\DeclareMathOperator{\Gal}{Gal}

\DeclareMathOperator{\Ker}{Ker}
\DeclareMathOperator{\Coker}{Coker}

\DeclareMathOperator{\trans}{trans}

\usepackage[all]{xy}

\begin{document}

\title[The $S$-relative P\'olya group and $S$-Ostrowski quotient]{The $S$-relative P\'olya groups and $S$-Ostrowski quotients of number fields}

\author[E. Shahoseini]{Ehsan Shahoseini}
\address{Department of Mathematics, Tarbiat Modares University, 14115-134, Tehran, Iran}
\curraddr{} \email{ehsan\_shahoseini@modares.ac.ir}
\thanks{}

\author[A. Maarefparvar]{Abbas Maarefparvar$^{*}$}
\address{Department of Mathematics and Computer Science,
	University of Lethbridge,
	Lethbridge, Alberta, T1K 3M4, Canada}
\curraddr{}
\email{abbas.maarefparvar@uleth.ca}
\thanks{$^{*}$Corresponding author}
\date{}
\dedicatory{}
\commby{}


\begin{abstract}
Let $K/F$ be a finite extension of number fields and $S$ be a finite set of primes of $F$, including all the archimedean ones. In this paper, using some results of Gonz\'alez-Avil\'es \cite{Aviles}, we generalize the notions of  the relative P\'olya group $\Po(K/F)$ \cite{ChabertI,MR2} and the Ostrowski quotient $\Ost(K/F)$ \cite{SRM} to their $S$-versions. Using this approach, we obtain generalizations of some well-known results on the $S$-capitulation map, including an $S$-version of  Hilbert's theorem 94.
\end{abstract}

\maketitle

\vspace{.2cm} {\noindent \bf{Keywords:}}~   Ostrowski quotient, relative P\'olya group, capitulation problem, BRZ exact sequence,  transgressive ambiguous classes.

\vspace{.2cm} {\noindent \bf{Notations.}}~  The following notations will be used throughout this article:

For a number field $K$, the notations $I(K)$, $P(K)$, $\Cl(K)$,  $\mathcal{O}_K$, $h_K$, $U_K$, $H(K)$, $\Gamma(K)$ and $\mathbb{P}_K$  denote the group of fractional ideals, group of principal fractional ideals, ideal class group, ring of integers, class number, group of units, Hilbert class field,  genus field of $K$ and set of all prime ideals of $K$, respectively.

For $K/F$ a finite extension of number fields, 
 we use  $N_{K/F}$ to denote both the ideal norm and the element norm map from $K$ to $F$. Also ${\epsilon}_{K/F}:\Cl(F) \rightarrow \Cl(K)$ denotes the \textit{capitulation map} induced by the extension homomorphism from $I(F)$ to $I(K)$. Finally, for a prime ideal $\mathfrak{p} \in \mathbb{P}_F$, we denote by $e_{\mathfrak{p}(K/F)}$ and  $f_{\mathfrak{p}(K/F)}$  the ramification index and residue class degree of $\mathfrak{p}$ in $K/F$, respectively.

\section{On Relative P\'olya groups and Ostrowski Quotients in finite extensions of number fields}

A number field $K$, with ring of integers $\mathcal{O}_K$, is called a \textit{P\'olya field} if for every prime number $p$ and every integer $f \geq 1$, the \textit{Ostrowski ideal}
\begin{equation} \label{equation, Ostrowski ideals}
	\Pi_{p^f}(K)=:\prod_{\substack{\mathfrak{p}\in \mathbb{P}_K \\ N_{K/\mathbb{Q}}(\mathfrak{p})=p^f}} \mathfrak{p}
\end{equation}
is principal (by convention, if $K$ has no ideal with norm $p^f$, we set $\Pi_{p^f}(K)=\mathcal{O}_K$) \cite{Zantema}.  As an obstruction measure for $K$ to be a P\'olya field, the notion of \textit{P\'olya group} was introduced later in \cite[$\S$II.4]{Cahen-Chabert's book}.
\begin{definition} \label{definition, Polya group}
	The P\'olya group of a number field $K$, denoted by $\Po(K)$, is the subgroup of the ideal class group $\Cl(K)$ generated by the classes of all the Ostrowski ideals $\Pi_{p^f}(K)$.
\end{definition}

Hence the P\'olyaness of $K$ is equivalent to the triviality of its P\'olya group. In particular, P\'olya fields encompass all class number one number fields. However, there are many P\'olya fields with non-trivial ideal class group; for instance, every cyclotomic field is a P\'olya field \cite[Proposition 2.6]{Zantema}. The reader is referred to \cite{ChabertI, MR1, MR2, Zantema} for some results on P\'olya fields and P\'olya groups.

The notion of P\'olya group has been recently generalized to the \textit{relative P\'olya group} in the following sense \cite{ChabertI, MR2}.

\begin{definition} \label{definition, relative Polya group} 
	Let $K/F$ be a finite extension of number fields. The \textit{relative P\'olya group} of $K$ over $F$, denoted by $\Po(K/F)$, is the subgroup of $\Cl(K)$ generated by the classes of the \textit{relative Ostrowski ideals}  
	\begin{align} \label{equation, relative Ostrowski ideal}
		\Pi_{\mathfrak{p}^f}(K/F):=\prod_{\substack{\mathfrak{P}\in \mathbb{P}_K \\ N_{K/ F}(\mathfrak{P})=\mathfrak{p}^f}} \mathfrak{P},
	\end{align}
	where $\mathfrak{p} \in \mathbb{P}_F$, $f$ is a positive integer. By convention, if $K$ has no ideal with relative norm $\mathfrak{P}^f$ (over $F$), then $\Pi_{\mathfrak{P}^f}(K/F)=\mathcal{O}_K$.
\end{definition}

\begin{remark}
	Note that $\Po(K/\mathbb{Q})=\Po(K)$ and $\Po(K/K)=\Cl(K)$. Also,
one can show that for $F \subseteq M \subseteq K$, a tower of finite extensions of number fields, if $K/M$ is Galois then $\Po(K/F) \subseteq \Po(K/M)$ \cite[Lemma 2.10]{MR2}. In this case, for the tower
	\begin{equation*}
		F=M_0 \subseteq M_1=M  \subseteq \dots \subseteq M_{n-1}\subseteq M_n=K,
	\end{equation*}
	we get the \textit{filtration}
	\begin{equation*}
		\Po(K/F) \subseteq \Po(K/M) \subseteq \dots \subseteq \Po(K/M_{n-1}) \subseteq \Po(K/K)=\Cl(K)
	\end{equation*}
	between $\Po(K/F)$ and $\Cl(K)$, or equivalently a filtration on $\frac{\Cl(K)}{\Po(K/F)}$. For instance, for a prime ideal $\mathfrak{p}$ of $F$, let $\mathfrak{P}$ be a prime ideal of $K$ above $\mathfrak{p}$. Denote by $Z(\mathfrak{P}/\mathfrak{p})$ and $T(\mathfrak{P}/\mathfrak{p})$ the \textit{decomposition group} and the \textit{inertia group} of $K/F$ at $\mathfrak{P}$, respectively \cite[Chapter 1]{NChildress}. By the Galois correspondence, let $K_Z$ be the fixed field of $Z(\mathfrak{P}/\mathfrak{p})$, and $K_T$ be the fixed field of $T(\mathfrak{P}/\mathfrak{p})$. Then for the tower $F \subseteq K_Z \subseteq K_T \subseteq K$, we get the above filtration.
\end{remark}

In the case that $K/F$ is a finite Galois extension, the relative P\'olya group $\Po(K/F)$ coincides with the group of strongly ambiguous ideal classes of $K/F$ \cite[Section 2]{MR2}. Moreover, using some  cohomological results of Brumer and Rosen \cite[$\S$2]{Brumer-Rosen} and Zantema \cite{Zantema}  one may find the following exact sequence (``BRZ'' stands for these authors).

\begin{theorem} \cite[Theorem 2.2]{MR2} \label{theorem, generalization of the Zantema's exact sequence}
	Let $K/F$ be a finite Galois extension of number fields with Galois group $G$. Then the following sequence is exact:
		\begin{equation*} \label{equation, BRZ exact sequence}
				\tag{BRZ} 
			0 \rightarrow \Ker({\epsilon}_{K/F}) \xrightarrow{\theta_{K/F}} H^1(G,U_K) \rightarrow \bigoplus_{\mathfrak{p} \in \mathbb{P}_F} \frac{\mathbb{Z}}{e_{\mathfrak{p}(K/F)}\mathbb{Z}} \rightarrow \frac{\Po(K/F)}{{\epsilon}_{K/F}(\Cl(F))} \rightarrow 0,
	\end{equation*}
	where $\epsilon_{K/F}:\Cl(F) \rightarrow \Cl(K)$ denotes the capitulation map, and $e_{\mathfrak{p}(K/F)}$ denotes the ramification index of the prime $\mathfrak{p} \in \mathbb{P}_F$ in $K/F$.
\end{theorem}

Some interesting consequences of Theorem \ref{theorem, generalization of the Zantema's exact sequence} can be found in \cite{ChabertI,MR2}.

\subsection{On the Ostrowski Quotient $\Ost(K/F)$}
It worth noting that, for $K/F$ a finite Galois extension of number fields, by \eqref{equation, BRZ exact sequence} if $\epsilon_{K/F}\left(\Cl(F)\right)=0$, then there exists a surjective map from $\bigoplus_{\mathfrak{p} \in \mathbb{P}_F} \mathbb{Z}/e_{\mathfrak{p}(K/F)}\mathbb{Z}$ onto the relative P\'olya group $\Po(K/F)$. Roughly speaking, in this case, $\Po(K/F)$ is \textit{controlled} by ramification. However, this assertion may not hold, in general. For instance, for a regular prime number $p > 19$ and an integer $n >1$, let $F=\mathbb{Q}(\zeta_p)$ and $K=\mathbb{Q}(\zeta_{p^n})$, where $\zeta_p$ (resp. $\zeta_{p^n}$) denote the $p$-th (resp. $p^n$-th) primitive root of unity. Then $\Po(K/F) \neq 0$, while the only map from $\bigoplus_{\mathfrak{p} \in \mathbb{P}_F} \mathbb{Z}/e_{\mathfrak{p}(K/F)}\mathbb{Z}$ to $\Po(K/F)$ is the zero map, see \cite[Example 3.1]{SRM} for the details. Motivating by this example, the authors and  A. Rajaei \cite[Section 3]{SRM} introduced the notion of Ostrowski quotient.

\begin{definition}  \label{definition, Ostrowski quotient}
Let $K/F$ be a finite extension of number fields. The \textit{Ostrowski quotient} of $K$ over $F$, denoted by $\Ost(K/F)$,  is defined as  
	\begin{equation} \label{equation, Ostrowski group general case}
		\Ost(K/F) := \dfrac{\Po(K/F)}{\Po(K/F) \cap \epsilon_{K/F}(Cl(F))}.
	\end{equation} 
	In particular, $\Ost(K/\mathbb{Q})=\Po(K/\mathbb{Q})=\Po(K)$ and $\Ost(K/K)= 0$. The extension $K/F$ is called ``\textit{Ostrowski}''  (or $K$ is called $F$-Ostrowski) if $\Ost(K/F)=0$.
\end{definition}

\begin{remark}
Note that $K/F$ is a finite Galois extension of number fields, then $\epsilon_{K/F}\left(\Cl(F)\right) \subseteq \Po(K/F)$ and \eqref{equation, BRZ exact sequence} can be re-written as
	\begin{equation*} \label{equation, BRZ exact sequence2}
		0 \rightarrow \Ker({\epsilon}_{K/F}) \xrightarrow{\theta_{K/F}} H^1(\Gal(K/F),U_K) \rightarrow \bigoplus_{\mathfrak{p} \in \mathbb{P}_F} \frac{\mathbb{Z}}{e_{\mathfrak{p}(K/F)}\mathbb{Z}} \rightarrow \Ost(K/F) \rightarrow 0.
	\end{equation*}
The reader is referred to \cite[Section 3]{SRM} for some results on Ostrowski quotients.
\end{remark}

\section{The $S$-versions of relative P\'olya group and Ostrowski Quotient} \label{section, S-capitulation}

Throughout this section, we fix the following notations:

\begin{itemize}
	\item $K/F$: a finite extension of number fields;
	
	\item
	$S_{\infty}$: the set of all the archimedean primes of $F$;
	
	\vspace*{0.15cm}
	
	\item
	$S$: a finite set of primes of $F$ containing $S_{\infty}$ (We also denote by $S$ the set of the primes of $K$ which lie above the primes in $S$);
	
	\vspace*{0.15cm}

	\item $U_{F,S}$ (resp. $U_{K,S}$) : the group of $S$-units of $F$ (resp. $K$);
	\vspace*{0.15cm}

	\item
	$I(F)_S$ (resp. $I(K)_S$) : the group of fractional ideals of $F$ (resp. $K$) with support outside $S$;
	\vspace*{0.15cm}
	
		\item
	$P(F)_S$ (resp. $P(K)_S$) : the group of principal fractional ideals of $F$ (resp. $K$) with support outside $S$;
	\vspace*{0.15cm}

	\item
	$\Cl(F)_S$ (resp. $\Cl(K)_S$) : the $S$-ideal class group of $F$ (resp. $K$).
	\vspace*{0.15cm}

	
	
\end{itemize}


In the case that $K/F$ is Galois,  Gonz\'alez-Avil\'es \cite{Aviles} found some interesting results on the kernel and cokernel of the \textit{$S$-capitulation map} 
\begin{equation} \label{definition, S-capitulation map}
	\epsilon_{K/F,S}:  \Cl(F)_S \rightarrow  \left(\Cl(K)_S\right)^G
\end{equation}
 where $G=\Gal(K/F)$. A part of his cohomological tools has been restated as follows.

\begin{proposition} \cite[Section 2]{Aviles} \label{proposition, Aviles results}
Let $K/F$ be a finite Galois extension of number fields with Galois group $G$. With the notations of this section, the following sequences are exact:
	\begin{align}
		& 0 \rightarrow \left(P(K)_S\right)^G\rightarrow \left(I(K)_S\right)^G \rightarrow \left(\Cl(K)_S\right)_{\trans}^G \rightarrow 0, \label{equation, ES-PKG to IKG to ClG-trans}
		\\
		&0 \rightarrow \left(\Cl(K)_S\right)_{\trans}^G \rightarrow \left(\Cl(K)_S\right)^G  \rightarrow  H^1(G, P(K)_S) \rightarrow 0, \label{equation, ES-ClG-trans to ClG}
		\\
		& 0  \rightarrow H^1(G,P(K)_S) \rightarrow H^2(G,U_{K,S}) \rightarrow H^2(G,K^{\times}). \label{equation, ES-H1(G,PK) to H2(G,UK)}
	\end{align}
\end{proposition}

\begin{remark}
Using the exact sequences \eqref{equation, ES-ClG-trans to ClG} and \eqref{equation, ES-H1(G,PK) to H2(G,UK)}, we obtain the  map 
	\begin{equation} \label{equation, S-trans map}
\trans_{K/F,S}: \left(\Cl(K)_S\right)^G   \rightarrow  H^2(G,U_{K,S})
	\end{equation} 
which is called the \textit{transgression map}. Also its kernel, i.e., $ \left(\Cl(K)_S\right)_{\trans}^G $, is called the group of \textit{transgressive ambiguous classes} \cite[Section 2]{Aviles}.
\end{remark}

As mentioned before, the relative P\'olya group $\Po(K/F)$ coincides with the group of strongly ambiguous ideal classes in $K/F$, i.e., $\Po(K/F)=I(K)^G/P(K)^G$ \cite[$\S$ 2]{MR2}. Hence using the exact sequence \eqref{equation, ES-PKG to IKG to ClG-trans} for $S=S_{\infty}$ we obtain
\begin{equation} \label{equation, relative Polya group equals to transgressive group}
\Po(K/F)=(\Cl(K)_{S_{\infty}})^G_{trans}.
\end{equation}
This leads us to define the notion of \textit{$S$-relative P\'olya group} in the following sense.

\begin{definition} \label{definition, S-relative Polya group}
	Let $K/F$ be a finite extension of number fields, and $S$ be a finite set of the primes of $F$, containing all the archimedean ones. The \textit{$S$-relative P\'olya group of $K$ over $F$}, denoted by $\Po(K/F)_S$, is the subgroup of the ideal class group of $K$ generated by  all the classes of the relative Ostrowski ideals $\Pi_{\mathfrak{p}^f}(K/F)$ \eqref{equation, relative Ostrowski ideal} with support outside $S$:
	\begin{equation*}
		\Po(K/F)_S:=\left< \left[ \Pi_{\mathfrak{p}^f}(K/F) \right] \, : \, \mathfrak{p} \in \mathbb{P}_{F} \backslash S, \, f \in \mathbb{N} \right>.
	\end{equation*}
\end{definition}

\begin{remark} 
Let $S_1 \subseteq S_2$ be two finite sets of the primes of $F$, containing $S_{\infty}$. Then by the above definition, it immediately follows that
	\begin{equation*}
	\Po(K/F)_{S_2} \subseteq \Po(K/F)_{S_1}.
	\end{equation*}
	 In particular,  $\Po(K/F)_{S_{\infty}}=\Po(K/F)$  is the maximal element in the family
	 \begin{equation*}
	 \{\Po(K/F)_S  \, : \,  \text{$S$ a finite set of the primes of $F$ with $S_{\infty} \subseteq S$} \}
	 \end{equation*}
	  of the subgroups of $\Cl(K)$. 
\end{remark}


\begin{theorem} \label{theorem, free generators for I(K)SG}
Let $K/F$ be a finite extension of number fields with Galois group $G$. Let $S$ be a finite set of the primes of $F$, containing all the archimedean ones. Denote also by $S$ the set of the primes of $K$ which lie above the primes in $S$. Then 
	\begin{equation*}
		\{ \Pi_{\mathfrak{p}^f}(K/F) \, : \, \mathfrak{p} \in \mathbb{P}_{F}\backslash S, \, f \in \mathbb{N}\}
	\end{equation*}	
	is a set of free generators for $\left(I(K)_S\right)^G$. 
\end{theorem}

\begin{proof}
	For a prime $\mathfrak{p} \in \mathbb{P}_{F}$, consider its decomposition in $K/F$ as 
	\begin{equation*}
		\mathfrak{p} \mathcal{O}_K=\left( \mathfrak{P}_1 \dots \mathfrak{P}_g\right)^{e_{\mathfrak{p}(K/F)}}=\left(\Pi_{\mathfrak{p}^{f_{\mathfrak{p}(K/F)}}}(K/F)\right)^{e_{\mathfrak{p}(K/F)}},	
	\end{equation*}
where $e_{\mathfrak{p}(K/F)}$ and  $f_{\mathfrak{p}(K/F)}$ denote the ramification index and the residue class degree of $\mathfrak{p}$ in $K/F$, respectively. Since the Galois group $G$ permutes the primes $\mathfrak{P}_i$'s transitively, we have
\begin{equation*}
	\left(\Pi_{\mathfrak{p}^{f_{\mathfrak{p}(K/F)}}}(K/F)\right)^{\sigma}=\Pi_{\mathfrak{p}^{f_{\mathfrak{p}(K/F)}}}(K/F) , \quad \forall \sigma \in G.
\end{equation*}
Hence if $ \mathfrak{p} \not \in S$, then $\Pi_{\mathfrak{p}^{f_{\mathfrak{p}(K/F)}}}(K/F) \in \left(I(K)_S\right)^G$. Now let $\mathfrak{a} \in \left(I(K)_S\right)^G$. For each prime $\mathfrak{P}$ of $K$ dividing $\mathfrak{a}$, let
\begin{equation*}
	\mathfrak{a}=\mathfrak{P}^{\nu_{\mathfrak{P}}(\mathfrak{a})}.\mathfrak{m}, \quad \mathfrak{P}  \nmid \mathfrak{m}.
\end{equation*}
For every $\sigma \in G$, we have
\begin{equation*}
\mathfrak{a}=\mathfrak{a}^{\sigma}=\left(\sigma(\mathfrak{P})\right)^{\nu_{\mathfrak{P}}(\mathfrak{a})}.\sigma(\mathfrak{m}),
\end{equation*}
which implies that the prime ideal $\sigma(\mathfrak{P})$ is also dividing $\mathfrak{a}$ with the same exponent $\nu_{\mathfrak{P}}(\mathfrak{a})$. Therefore 
\begin{equation*}
\mathfrak{a}=\prod_{\substack{\mathfrak{P}\in \mathbb{P}_K, \, \mathfrak{P} \mid \mathfrak{a} \\ \mathfrak{P} \cap F=\mathfrak{p}}}\left(\Pi_{\mathfrak{p}^{f_{\mathfrak{p}(K/F)}}}(K/F)\right)^{\nu_{\mathfrak{P}}(\mathfrak{a})},
\end{equation*}
and the assertion is proved, since $\mathfrak{a} \in \left(I(K)_S\right)^G$.
\end{proof}

\begin{corollary} \label{corollary, S-RPG coincides with Cl(K)SGtrans}
	Let $K/F$ be a finite extension of number fields with Galois group $G$. Let $S$ be a finite set of the primes of $F$, containing all the archimedean ones. Then
	\begin{equation} \label{equation, S-RPG coincides with Cl(K)SGtrans}
		\Po(K/F)_S=\left(\Cl(K)_S\right)^G_{\trans}.
	\end{equation}
\end{corollary}

\begin{proof}
By Theorem \ref{theorem, free generators for I(K)SG} we have $\Po(K/F)_S=\left(I(K)_S\right)^G/\left(P(K)_S\right)^G$. Now the assertion immediately follows from the exact sequence \eqref{equation, ES-PKG to IKG to ClG-trans}.
\end{proof}

\subsection{$S$-Ostrowski quotient}
Recall that, as a modification of the notion of the relative P\'olya group, the notion of Ostrowski quotient $\Ost(K/F)$ has been recently defined in \cite{SRM}, see Definition \ref{definition, Ostrowski quotient}.
Similar to the $S$-relative P\'olya group, one can naturally define the $S$-relative version of the Ostrowski quotient. 

\begin{definition}  \label{definition, S-relative Ostrowski group}
	Let $K/F$ be a finite extension of number fields, and $S$ be a finite set of the primes of $F$, containing all the archimedean ones. The \textit{$S$-relative Ostrowski quotient}  $\Ost(K/F)_S$  is defined as  
	\begin{equation} \label{equation, Ostrowski group general case}
		\Ost(K/F)_S := \dfrac{\Po(K/F)_S}{\Po(K/F)_S \cap \epsilon_{K/F,S}(Cl(F)_S)}.
	\end{equation} 
	In particular, $\Ost(K/F)_{S_{\infty}}=\Ost(K/F)$. 
	The extension $K/F$ is called \textit{$S$-Ostrowski}  (or $K$ is called \textit{$(F,S)$-Ostrowski}) if $\Ost(K/F)_S$ is trivial.
\end{definition}

\begin{remark}
	Let $K/F$ be a finite Galois extension with Galois group $G$. Then
	\begin{equation*}
		\epsilon_{K/F,S}\left(Cl(F)_S\right) \subseteq \left(\Cl(K)_S\right)^G_{trans}=\Po(K/F)_S,
	\end{equation*}
	where  
	\begin{equation*}
		\epsilon_{K/F,S}:  \Cl(F)_S \rightarrow  \left(\Cl(K)_S\right)^G
	\end{equation*}
	denotes the $S$-capitulation map. Hence by considering the map
	\begin{equation} \label{equation, s-capitulation map prime}
		\epsilon_{K/F,S}^{\prime}:  \Cl(F)_S \rightarrow \Po(K/F)_S,
	\end{equation}
	induced by $\epsilon_{K/F,S}$, we have
	\begin{equation} \label{equation, S-relative Ostrowski quotient coincides with cokernel of epsilon'}
		\Ost(K/F)_S=\Coker \left(\epsilon^{\prime}_{K/F,S}\right).
	\end{equation}
	
\end{remark}

\subsection{The $S$-BRZ exact sequence}
One can restate some results of Avil\'es  \cite{Aviles} in terms of the $S$-relative P\'olya group and the $S$-relative Ostrowski quotient.

\begin{proposition}\cite[Proposition 2.2]{Aviles}
	For $K/F$ a finite Galois extension with Galois group $G$, the following sequence is exact	
	\begin{equation*}
		0 \rightarrow \Ost(K/F)_S \rightarrow \Coker(\epsilon_{K/F,S}) \rightarrow H^1(G,P(K)_S) \rightarrow  0.
	\end{equation*}
	In particular, for $S=S_{\infty}$ we obtain the exact sequence
	\begin{equation*}
		0 \rightarrow 	\Ost(K/F) \rightarrow \frac{\Cl(K)^G}{\epsilon_{K/F}\left(\Cl(F)\right)} \rightarrow H^1(G,P(K)) \rightarrow  0,
	\end{equation*}
	where $P(K)$ denotes the group of principal fractional ideals of $K$.
\end{proposition}

\begin{lemma} \cite[Lemma 2.2]{Aviles} \label{lemma, Z/e_vZ is isomorphic to H^1(G_w,U_w)}
	Let $K/F$ be a finite Galois extension of number fields with Galois group $G$. Let $\mathfrak{p}$ be a non-archimedean prime of $F$ which is ramified in $K$. Then there exists a canonical isomorphism
	\begin{align*}
		&\mathbb{Z}/e_{\mathfrak{p}(K/F)}\mathbb{Z} \simeq H^1(G_{\mathfrak{P}},U_{K_{\mathfrak{P}}}), 
	\end{align*}
	where  $e_{\mathfrak{p}(K/F)}$ denotes the ramification index of $\mathfrak{p}$ in $K$, $\mathfrak{P}$ is a prime of $K$ above $\mathfrak{p}$, $K_{\mathfrak{P}}$ (resp. $F_{\mathfrak{p}}$) is the localizations of $K$ (resp. $F$) at $\mathfrak{P}$ (resp. at $\mathfrak{p}$), $G_{\mathfrak{P}}$ is the Galois group $\Gal(K_{\mathfrak{P}}/F_{\mathfrak{p}})$ and $U_{K_{\mathfrak{P}}}$ is the group of units in $K_{\mathfrak{P}}$.  
\end{lemma}



\begin{theorem} \cite[Theorem 2.4]{Aviles} \label{theorem, Aviles exact sequence}
	For $K/F$ a finite Galois extension of number fields with Galois group $G$, there exists an exact sequence
		\begin{equation} \label{equation, Aviles exact sequence}
			\xymatrix{
				0 \rightarrow    \Ker({\epsilon}_{K/F,S}) \rightarrow   H^1(G,U_{K,S}) \xrightarrow{\lambda}  \bigoplus_{\mathfrak{p} \in R_{K/F} \backslash S} H^1(G_{\mathfrak{P}},U_{K_{\mathfrak{P}}}) \rightarrow \Ost(K/F)_S  \rightarrow  0,
			}
	\end{equation}
	where $ \epsilon_{K/F,S}:  \Cl(F)_S \rightarrow  \left(\Cl(K)_S\right)^G$ denotes the $S$-capitulation map, $R_{K/F}$ denotes the set of all non-archimedean primes of $F$  ramified in $K$, $\mathfrak{P}$ is a prime of $K$ above $\mathfrak{p}$, and $G_{\mathfrak{P}}=\Gal(K_{\mathfrak{P}}/F_{\mathfrak{p}})$ for $K_{\mathfrak{P}}$ (resp.  $F_{\mathfrak{p}}$)  the localization of $K$ (resp. $F$) at $\mathfrak{P}$ (resp. at $\mathfrak{p}$).
\end{theorem}


\begin{remark}
	By Lemma \ref{lemma, Z/e_vZ is isomorphic to H^1(G_w,U_w)} the exact sequence \eqref{equation, Aviles exact sequence} can be  re-written as
	\begin{equation} \label{equation, S-BRZ}
			\tag{S-BRZ}
			\xymatrix{
				0 \rightarrow    \Ker({\epsilon}_{K/F,S}) \rightarrow   H^1(G,U_{K,S}) \xrightarrow{\theta_{K/F,S}}  \bigoplus_{\mathfrak{p} \in R_{K/F} \backslash S} \frac{\mathbb{Z}}{e_{\mathfrak{p}(K/F)}\mathbb{Z}}\rightarrow \Ost(K/F)_S  \rightarrow  0,
			}
	\end{equation}
which is a vast generalization of  \eqref{equation, BRZ exact sequence}. Indeed, the above exact sequence can be thought of as the ``\textit{$S$-relative version of \eqref{equation, BRZ exact sequence}}'' which we call ``\textit{$S$-BRZ}''.
\end{remark}

\subsection{Some applications of $S$-BRZ}
In \cite[$\S$2]{SRM}, the authors and A. Rajaei used the exact sequence \eqref{equation, BRZ exact sequence} to give short and simple proofs for some well-known results in the literature. Using \eqref{equation, S-BRZ},  one can obtain these results in a more general setting, namely in their $S$-relative versions. 


\begin{theorem} \label{theorem, Iwasawa-khare-Prasad for S-integers}
	\textbf{(generalizing Iwasawa-Khare-Prasad result \cite{Iwasawa,Khare})} Let $K/F$ be a  finite Galois extension of number fields and $S$ be a finite  set of primes of $F$ with $S_{\infty} \subseteq S$.
	If $K/F$ is unramified outside $S$, then 
	\begin{equation*}
		\Ker(\epsilon_{K/F,S}) \simeq H^1(\Gal(K/F),U_{K,S}), \quad \text{and} \qquad \Ost(K/F)_S=0.
	\end{equation*}
\end{theorem}

\begin{proof}
	Immediately follows from \eqref{equation, S-BRZ}.
\end{proof}

\begin{corollary}
	Let $K/F$ be a finite cyclic extension of number fields and $S$ be a finite  set of primes of $F$ with $S_{\infty} \subseteq S$. If $K/F$ is unramified outside $S$, then $H^2(G,U_{K,S}) \simeq \Coker(\epsilon_{K/F,S})$.
\end{corollary}

\begin{proof}
	Using the exact sequence \eqref{equation, ES-ClG-trans to ClG} and Corollary \ref{corollary, S-RPG coincides with Cl(K)SGtrans} we get the following exact sequence
	\begin{equation} \label{equation, sequence: Po(K/F)S to Cl(K)SG to H1(G,P(K)S)}
		0 \rightarrow \Po(K/F)_S \rightarrow  \left(\Cl(K)_S\right)^G \rightarrow H^1(G,P(K)_S) \rightarrow  0.
	\end{equation}
	Since $K/F$ is cyclic and unramified outside $S$, one can easily obtain the ``$S$-version'' of Kisilevsky's result in \cite[Lemma 1]{Kisilevsky}:
	\begin{equation} \label{equation, S-version of Kisilevsky result}
		H^1(G,P(K)_S) \simeq H^2(G,U_{K,S}).
	\end{equation}
	By Theorem \ref{theorem, Iwasawa-khare-Prasad for S-integers} one has $\Ost(K/F)_S=0$, i.e., $\Po(K/F)_S =\epsilon_{K/F,S}\left(\Cl(F)_S\right)$. Using the relations \eqref{equation, sequence: Po(K/F)S to Cl(K)SG to H1(G,P(K)S)} and \eqref{equation, S-version of Kisilevsky result}, the proof is completed.
\end{proof}

\begin{theorem} \label{theorem, Iwasawa result for S-integers}
	\textbf{(generalizing a result of Iwasawa \cite{Iwasawa})} Let $K/F$ be finite Galois extension of number fields. For  a modulus   $\mathfrak{m}$ of $F$, let  $F^{\mathfrak{m}}$ be the ray class field of $F$ of modulus $\mathfrak{m}$, and $S$ be the support of $\mathfrak{m}$ along with the set of all archimedean places of $F$. If $F^{\mathfrak{m}} \subseteq K$ and $K/F$ is unramified outside $S$, then
	\begin{equation*}
		H^1(\Gal(K/F),U_{K,S}) \simeq \Cl(F)_{S},
	\end{equation*}
	where $\Cl(F)_{S}$ denotes the $S$-ideal class group of $F$.
\end{theorem}

\begin{proof}
	Since $K/F$ is unramified outside $S$, the exact sequence \eqref{equation, S-BRZ} implies that 
	\begin{equation} \label{equation, H^1(Gal,UK,S) is isomorphic to Ker(epsilon,S)}
		H^1(\Gal(K/F),U_{K,S}) \simeq \Ker(\epsilon_{K/F,S}).
	\end{equation}
	Since $F^{\mathfrak{m}} \subseteq K$, $\epsilon_{K/F,S}=\epsilon_{K/F^{\mathfrak{m}},S} \circ \epsilon_{F^{\mathfrak{m}}/F,S}$, by the principal ideal theorem for ray class fields, we have 
	\begin{equation*}
		\epsilon_{F^{\mathfrak{m}}/F,S} \left(\Cl(F)_S \right) =0,
	\end{equation*}
	which implies that $\epsilon_{K/F,S}\left( (\Cl(F)_S  \right) =0$. Hence $\Ker(\epsilon_{K/F,S}) = \Cl(F)_S$ and using the equation \eqref{equation, H^1(Gal,UK,S) is isomorphic to Ker(epsilon,S)} we obtain the desired isomorphism.
\end{proof}

\begin{theorem} \label{theorem, generalizing Tannaka's theorem to ray class fields}
	\textbf{(generalizing a result of Tannaka \cite[Theorem 8]{Tannaka})}	For a number field $F$ and  a modulus  $\mathfrak{m}$ of $F$, let  $F^{\mathfrak{m}}$ be the ray class field of $F$ of modulus $\mathfrak{m}$, and $S$ be the support of $\mathfrak{m}$ along with the set of all archimedean places of $F$. 
	Then there exists a surjective map
	\begin{equation*}
		\Gal(F^{\mathfrak{m}}/F) \twoheadrightarrow 	H^1(\Gal(F^{\mathfrak{m}}/F),U_{F^{\mathfrak{m}},S}).
	\end{equation*}
\end{theorem}

\begin{proof}
	By Theorem \ref{theorem, Iwasawa result for S-integers} we have
	\begin{equation*} 
		\Cl(F)_S \simeq H^1(\Gal(F^{\mathfrak{m}}/F),U_{F^{\mathfrak{m}},S}).
	\end{equation*}
	
	On the one hand, by the \textit{Artin Reciprocity} one has
	\begin{equation*}
		\Gal(F^{\mathfrak{m}}/F) \simeq  \Cl(F)_{\mathfrak{m}},
	\end{equation*}	
	where $\Cl(F)_{\mathfrak{m}}$ denotes the 
		\textit{ray class group} of $F$ for $\mathfrak{m}$, i.e. the group $I(F)_{\mathfrak{m}}/P(F)_{\mathfrak{m}}$ for $I(F)_{\mathfrak{m}}$ the subgroup of $I(F)$ that do not involve the primes dividing $\mathfrak{m}$, and $P(F)_{\mathfrak{m}} \subseteq I(F)_{\mathfrak{m}}$ the group of principal ideals generated by some \textit{totally positive} elements $a \in \mathcal{O}_F$ with $a \equiv 1 (\mathrm{mod} \, \mathfrak{m})$, see \cite[Chapter 3]{NChildress}.

		On the other hand, there exists a surjective map $\Cl(F)_{\mathfrak{m}} \twoheadrightarrow \Cl(F)_S$ and the proof is completed.
	\end{proof}

	\begin{theorem} \label{theorem, a generalization of Hilbert 94 for S-integers}
		\textbf{(a generalization of Hilbert's theorem 94 \cite{Hilbert's book})} Let $K/F$ be a finite cyclic extension of number fields with Galois group $G$, and $S$ be a finite  set of primes of $F$ with $S_{\infty} \subseteq S$. If $K/F$ is unramified outside $S$, then 
		\begin{equation*}
			\left(\prod_{v \in S} \left|G_w\right|\right).\left|\Ker(\epsilon_{K/F,S})\right|=\left(U_{F,S}:N_{K/F}(U_{K,S})\right).[K:F],
		\end{equation*}
		where $G_w$ denotes the decomposition group of $G$ at $w$ for a prime $w$ of $K$ above $v \in S$.
	\end{theorem}

	\begin{proof}
	On the one hand, since $K/F$ is unramified outside $S$, the exact sequence \eqref{equation, S-BRZ} implies that
		\begin{equation*}
			\left| H^1(G,U_{K,S}) \right|=\left| \Ker(\epsilon_{K/F,S}) \right|.
		\end{equation*}
On the other hand,	since $K/F$ is cyclic, one can use the \textit{Herbrand quotient}
		\begin{equation*}
			Q(G,U_{K,S})=\frac{\left|\widehat{H^0}(G,U_{K,S})\right|}{\left|H^1(G,U_{K,S})\right|},
		\end{equation*}
		where 
		\begin{equation*}
			\widehat{H^0}(G,U_{K,S})=\frac{U_{K,S}^G}{N_{K/F}(U_{K,S})}=\frac{U_{F,S}}{N_{K/F}(U_{K,S})},
		\end{equation*}
		and
		\begin{equation*}
			Q(G,U_{K,S})=\frac{\prod_{v \in S}\left|G_w\right|}{[K:F]},
		\end{equation*}
		see \cite[Proof of Proposition 5.10]{NChildress}. The above equalities give the desired result.
	\end{proof}

	\begin{remark}
As the above results demonstrate, one may investigate finding the $S$-relative versions of the results presented in \cite[Section 3]{SRM}. However, there are obstacles in some special cases. For instance, the author and A. Rajaei proved that ``For $K/F$, a finite cyclic extension, we have $\Ost(\Gamma(K/F)/F)=0$, where $\Gamma(K/F)$ denotes the relative genus field of $K$ over $F$'' \cite[Theorem 3.19]{SRM}. However, for obtaining the $S$-version of this result, we need to define the notion of the $S$-relative genus field (in the narrow sense) and get the $S$-version of Terada's ``Principal Ideal Theorem'' \cite{Terada}. These will be investigated in future works.	
	\end{remark}

\bibliographystyle{amsplain}

\end{document}